\documentclass[12pt, a4paper, leqno]{scrartcl}

\usepackage[T1]{fontenc}					
\usepackage[english]{babel}					
\usepackage[utf8]{inputenc}				

\usepackage{array}

\usepackage{amsmath}
\usepackage{amssymb}
\usepackage{amsfonts}
\usepackage[intlimits]{empheq}
\usepackage{amsthm}
\usepackage{xfrac}
\usepackage{mathtools}
\usepackage{ulem}

\usepackage{enumerate}

\usepackage{authblk}

\usepackage{graphicx}

\usepackage{hyperref}
\usepackage[all]{hypcap}

\usepackage{xcolor}

\usepackage{csquotes}

\newcommand{\hide}[1]{}

\title{\Large Universal locally univalent functions and universal conformal metrics with constant curvature}

\author{\large Daniel Pohl\thanks{\texttt{daniel.pohl@mathematik.uni-wuerzburg.de}} }
\author{\large Oliver Roth\thanks{\texttt{roth@mathematik.uni-wuerzburg.de},
    Phone: +49 931 318 4974} }
\affil{\small \textsc{Department of Mathematics, University of W\"urzburg\\
			Emil Fischer Stra\ss e 40, 97074 W\"urzburg, Germany}}

\date{\normalsize \today}

\allowdisplaybreaks

\numberwithin{equation}{section}

\newcommand{\N}{\mathbb{N}}
\newcommand{\Z}{\mathbb{Z}}

\newcommand{\R}{\mathbb{R}}
\newcommand{\C}{\mathbb{C}}
\newcommand{\D}{\mathbb{D}}

\newcommand{\limn}{\lim_{n\rightarrow\infty}}
\newcommand{\limm}{\lim_{m\rightarrow\infty}}

\newcommand{\dz}{\left|dz\right|}
\newcommand{\dist}{\mathrm{dist}}
\newcommand{\hol}{\mathcal{H}}

\newcommand{\M}{\mathcal{M}}
\newcommand{\B}{\mathcal{B}}
\renewcommand{\O}{\mathcal{O}}
\newcommand{\aut}[1]{\mathrm{Aut}\left(#1\right)}

\newtheoremstyle{plain1}{9pt}{9pt}%
                                {\itshape}%
                                {-0pt}%
                                {\sffamily\bfseries}{.}%
                                {\newline}%
                                {}%
\theoremstyle{plain1}
\newtheorem{thm}{Theorem}[section]					
\newtheorem{cor}[thm]{Corollary}					
\newtheorem{prop}[thm]{Proposition}			  
\theoremstyle{definition}										

\newtheorem{example}[thm]{Example}
\newtheorem{definition}[thm]{Definition}
\newtheorem{remark}[thm]{Remark}
\theoremstyle{plain1}\newtheorem{problem}[thm]{Problem}

\renewcommand\qedsymbol{\ensuremath{\blacksquare}}

\newenvironment{keywords}%
   {\begin{trivlist}\item[]{\bfseries\sffamily Keywords:}\ }
   {\end{trivlist}}
\newenvironment{msc}%
   {\begin{trivlist}\item[]\textit{2010 Mathematics Subject Classification:}\ }
   {\end{trivlist}}

\begin{document}

\maketitle

\begin{abstract}
	\textbf{\textsc{Abstract}}
	We prove Runge--type theorems and universality results for
         \textit{locally
        univalent} holomorphic and meromorphic functions. Refining a result of
        M.~Heins, we also show that there
        is a universal \textit{bounded} locally univalent function on the unit
        disk. These results are used to prove that
        on  any hyperbolic simply connected plane domain there exist universal
        conformal metrics with prescribed constant curvature.
\end{abstract}

\begin{keywords}
	Universal functions, Runge theory, locally univalent functions, conformal metrics, constant curvature
\end{keywords}
\begin{msc}
	30E10 $\cdot$ 30F45 $\cdot$ 30K20 $\cdot$ 30K99
\end{msc}

\section{Introduction}


Let $\Omega$ be a domain in the complex plane $\C$ and let $\hol(\Omega)$ be
the space of all holomorphic functions on $\Omega$. 
We think of $\hol(\Omega)$ as a (closed) subspace of the Fr\'echet space
$C(\Omega)$ of all complex--valued continuous functions on $\Omega$ equipped with the
compact--open topology of locally uniform convergence. We denote by 
 $\aut{\Omega}$ the group of all conformal automorphisms of $\Omega$.
A function $f\in\hol(\Omega)$ is called \textit{universal} 
if the set $\{f \circ \phi \, : \, \phi \in \aut{\Omega}\}$ is dense in
$\hol(\Omega)$, i.e., as big as it possibly can be.


\medskip

The concept of universality goes back at least to Birkhoff in 1929, who showed \cite{birkhoff1929demonstration}  that there exist
universal functions in $\hol(\C)$. Another early, related universality result
was obtained by  Seidel and\ Walsh in 1941, who proved in \cite{seidel1941approximation}  that there are
 universal functions in $\hol(\D)$, where $\D$ denotes the open unit
 disc. 

\medskip

On the other hand,  it follows from the maximum principle that there are
 no universal functions in $\hol(\C \setminus \{0\})$. In all other cases,
so when $\Omega$ is not conformally equivalent to $\C\setminus\{0\}$, then 
universal functions $f\in\hol(\Omega)$ exist if and only if $\aut{\Omega}$ is
not compact. These results are due to  Bernal--Gonz\'alez and Montes--Rodr\'iguez
\cite{bernal1995universal}. They fully characterize
 all domains $\Omega$ in $\C$  for which  universal functions $f\in\hol(\Omega)$ exist.


\medskip

The notion of universality has been modified for many other classes of
holomorphic and meromorphic functions and even beyond. We refer the reader to the survey paper
\cite{erdmann1999universal}, but  wish to explicitly point out 
three specific universality results:

\begin{itemize}
\item[(1)] (Universal bounded holomorphic functions, Heins \cite{heins1954universal})\\
There are universal \textit{bounded} holomorphic functions. In fact,  
 there is a universal Blaschke product $B$
 such that $\{B \circ \phi \, : \, \phi \in \aut{\D}\}$ is 
 dense in the set of all holomorphic self--maps of $\D$.

\item[(2)] (Universal univalent functions, Pommerenke \cite{pommerenke1964linear})\\
Let $S$ be the class of all univalent holomorphic functions $f\colon\D\rightarrow\C$ normalized by $f(0)=0$ and $f^{\prime}(0)=1$.
Then  there is a universal
function $f \in S$ in the sense that
$$\left\{ \frac{f \circ \phi-(f \circ \phi)(0)}{(f \circ \phi)'(0)} \, : \,
  \phi \in \aut{\D}\right\}$$
is dense in $S$. Note that the imposed normalization condition for the class
$S$ requires passing from $f \circ \phi$ to the ``Koebe
transforms'' $[f \circ \phi-(f \circ \phi)(0)]/(f \circ \phi)'(0)$.

\item[(3)] (Universal meromorphic functions)\\
There are Birkhoff--type universality results for meromorphic
functions (see e.g.~\cite{chan2001universal}).
\end{itemize}

The main goals of the present paper are to investigate universal
\textit{locally univalent} holomorphic and meromorphic functions 
 and, \textcolor{violet}{in particular},  \textit{universal constantly curved
conformal metrics}. A key auxiliary step consists in establishing \textit{Runge--type
results for locally univalent functions} which might be interesting in their
own right.

\medskip

Here is a quick outline of our work.
We denote for a set $M \subseteq \C$ by $\hol_{l.u.}(M)$ the
family of all  functions which are holomorphic and
locally univalent on some open neighborhood of $M$ in $\C$ (which might depend on the function).

\begin{thm}[Runge--type theorem for locally univalent holomorphic functions] \label{thm:r1}
Let $\Omega$ be a domain in $\C$ and let $K$ be a compact set in $\Omega$ such
that $\Omega \setminus K$ has no relatively compact components in $\Omega$.
Then every function $f \in \hol_{l.u.}(K)$ can be approximated
uniformly on $K$ by functions in $\hol_{l.u.}(\Omega)$.
\end{thm}
There is an analogue of Theorem \ref{thm:r1} for meromorphic
  functions, cf.~Theorem \ref{thm:loc_univ_mer_runge} below.

\begin{remark}
 It can be shown that Theorem \ref{thm:r1} also holds more generally on
  any open
  Riemann surface. Note that it is already a deep result due to Gunning 
  and Narasim\-han \cite{gunning1967immersion}  that every open Riemann
  surface carries at least one  locally univalent holomorphic function. Recently, Frostneri\v{c} \cite{forstnerivc2003noncritical} extended the
Gunning--Narasimhan theorem to Stein manifolds and Majcen \cite{majcen2007closed}
established a Runge--type theorem for holomorphic 1--forms on Stein manifolds.
We  shall use some of the ideas of these papers in  our proof of Theorem \ref{thm:r1}.
\end{remark}

\begin{thm}[Universal locally univalent functions] \label{thm:u1}
Let $\Omega$ be a simply connected domain in $\C$.
Then there exists a  function $f \in \hol_{l.u.}(\Omega)$ such that
$\{f \circ \phi \, : \, \phi \in \aut{\Omega}\}$ is dense in
$\hol_{l.u.}(\Omega)$.
\end{thm}

Hence there exist \textit{universal locally univalent holomorphic functions} on any simply
connected plane domain~$\Omega$. In fact, these functions 
form a dense $G_{\delta}$--subset of $\hol_{l.u.}(\Omega)$. A version of Theorem \ref{thm:u1} also holds for meromorphic
  functions, see Theorem \ref{thm:mer_loc_univ_birk} below.

\begin{remark}
 As it is the case with Theorem \ref{thm:r1}, also Theorem \ref{thm:u1} can be generalized to Riemann surfaces.
We note that Montes--Rodr\'iguez \cite{montes1998birkhoff} has studied
  universal holomorphic functions on Riemann surfaces, and
	his approach can be combined with Theorem \ref{thm:u1} to investigate
        universal locally univalent holomorphic functions on Riemann surfaces.
\end{remark}

We denote by $\mathcal{B}(\Omega)$ the set of all $f \in \hol(\Omega)$ such
that $|f(z)| \le 1$ on $\Omega$, and write
$ \mathcal{B}_{l.u.}(\Omega):=\mathcal{B}(\Omega) \cap \hol_{l.u.}(\Omega)$
for the set of bounded locally univalent functions.

\begin{thm}[Universal bounded locally univalent functions] \label{thm:u2}
Let $\Omega$ be a simply connected proper subdomain of $\C$.  
Then there is a function $f \in \mathcal{B}_{l.u.}(\Omega)$ such that
$\{ f \circ \phi \, : \, \phi \in \aut{\Omega}\}$ is dense in
$\mathcal{B}_{l.u.}(\Omega)$.
\end{thm}

\hide{\sout{Theorem \ref{thm:u2} is in the spirit of the 
 universality results of Heins
and Pommerenke mentioned above which are concerned with universal \textit{bounded}
resp.~\textit{univalent} functions.}}
Theorem \ref{thm:u2} is in the spirit of the 
 universality results of Heins
and Pommerenke mentioned above. However, Heins \cite{heins1954universal} considers only bounded, but not
necessarily \textit{locally univalent} functions, and Pommerenke \cite{pommerenke1964linear} is concerned with
univalent functions, which are not necessarily \textit{bounded}.
We note that while the proof of Theorem \ref{thm:u1}  is based on the
  Runge--type Theorem \ref{thm:r1},  the proof of
  Theorem \ref{thm:u2} in Section 2.3 below
is considerably different and will be based  on Heins' universality result and
the use of universal covering maps, see Theorem  \ref{thm:lu_universal_functions}.

\medskip
Theorem \ref{thm:u1} and  Theorem \ref{thm:u2} put us in a position to investigate
universal conformal metrics with constant curvature.
Recall (see Ahlfors \cite[\S 1.5]{ahlfors2010} or Simon \cite[Chapter 12]{simon2015})  that
a  regular conformal metric $\lambda(z) \, |dz|$ on a domain $\Omega$ is given by  a positive
$C^2$--function $\lambda$ on $\Omega$, called the density of the metric. 
Denoting, as usual,  by $\Delta:= \frac{\partial^2}{\partial
    x^2}+\frac{\partial^2}{\partial y^2}$ the Laplace operator  in the standard Cartesian coordinates of the $xy$--plane,
the
Gauss curvature of such a metric $\lambda(z) \, |dz|$, 
\begin{equation} \label{eq:curv}
 \kappa_{\lambda}(z):=-\frac{\Delta \log \lambda(z)}{\lambda(z)^2} \, ,
\qquad z\in \Omega \, ,\end{equation}
 has an important invariance property: If we define for 
a locally univalent self--map $\phi$ of $\Omega$ the pullback
$\phi^*\lambda(z)\, |dz|$ of
$\lambda(z) \, |dz|$ via $\phi$ by
$$\phi^*\lambda(z):=(\lambda \circ \phi)(z) \, |\phi'(z)| 
\, , \qquad z \in \Omega \, , $$
then
$$ \kappa_{\phi^*\lambda}=\kappa_{\lambda}\circ \phi \, .$$
Hence, for conformal metrics $\lambda(z) \, |dz|$ it is  more natural to consider the pullback
$\phi^*\lambda$ instead of the composition $\lambda \circ \phi$, even though
the additional ``conformal factor'' $|\phi' (z)|$ causes some difficulties.

\medskip
We denote by $\Lambda_c(\Omega)$ the set (of densities) of all regular conformal metrics with
constant curvature $c \in \R$. Then $\Lambda_c(\Omega)$ is a subset of the
Fr\'echet space $C(\Omega)$ which is invariant under pullback in the 
sense that
$\phi^*\lambda \in \Lambda_c$  for all $\lambda \in\Lambda_c$ and all locally
univalent self--maps $\phi$ of $\Omega$. We can now state the 
following theorem, which is maybe the  main result of this paper.

\begin{thm}[Universal constantly curved conformal metrics] \label{thm:u3}
Let  $\Omega$ be a simply connected domain in $\C$ and let $c \in \R$.
Suppose that $\Omega \not=\C$ if $c<0$. 
Then there is an   $\Lambda \in \Lambda_c(\Omega)$ such
  that $\{\phi^*\Lambda \, : \, \phi \in \aut{\Omega}\}$ is dense in
  $\Lambda_{c}(\Omega)$. 
\end{thm}

Hence there exist \textit{universal} constantly curved conformal metrics on any simply
connected domain $\not=\C$.  Theorem \ref{thm:u3} is perhaps the first universality
result for conformal metrics, except possibly for the case of constant curvature $c=0$
which is closely related to universal harmonic functions, see Theorem \ref{thm:harm_univ}
below for more details.

\medskip
We wish to point out that
  proving universality results for conformal metrics or -- and this amounts in
  view of (\ref{eq:curv}) to
  the same thing -- solutions of the Gauss curvature equation $\Delta u=-c\, 
  e^{2u}$, a basic nonlinear elliptic PDE in conformal geometry,  has been our main initial motivation for proving universality results for
  \textit{locally univalent} functions.

\medskip

The present paper is organized as follows. In Section 2 we discuss in greater
generality various Runge--type theorems and universality results for locally univalent holomorphic and
meromorphic functions as well as for constantly curved conformal
metrics. There are many related open problems  and we explicitly
discuss a considerable number of them. The proofs of the results are deferred to the final
Section 3.

\medskip

\textbf{Acknowldegements.} We wish to thank the referees
  for a very careful  
reading of the original manuscript and their many thoughtful suggestions which considerably
helped to improve the quality of this paper.


 \section{Results and open problems}
 
 In what follows $\Omega$ is always a domain in $\C$ and $\M(\Omega)$  the set of all meromorphic functions on $\Omega$.
 We think of $\M(\Omega)$ as a metric space equipped with the
 (metrizable)  topology of locally uniform convergence
 w.r.t.~the chordal metric $\chi$ on the Riemann sphere
 $\hat{\C}=\C\cup\{\infty\}$ where $\chi$ is defined as usual   by  
$$\chi(z_1,z_2):=\frac{|z_1-z_2|}{\sqrt{1+|z_1|^2} \sqrt{1+|z_2|^2}} \, , \quad
\text{ if } z_1,z_2 \in \C \, , $$
and $$ \chi(z_1,\infty):=\chi(\infty,z_1):=\frac{1}{\sqrt{1+|z_1|^2}} \, , \qquad
  \text{ if } z_1 \in \C \, .$$

For $f_n,f\in\M(\Omega)$ we write $f_n \to f$ locally $\chi$--uniformly
on $\Omega$ if $\chi(f_n,f) \to 0$
 locally uniformly on $\Omega$. 
 %

\medskip

 A function $f\in\M(\Omega)$ is called \textit{locally univalent} if $f$ has
 at most simple poles
 and $f^{\prime}(z)\neq0$ for all $z\in\Omega$ with $f(z)\neq\infty$.
For a family $\mathcal{G}\subseteq\M(\Omega)$ we let
$$\mathcal{G}_{l.u.}\coloneqq\{f\in\mathcal{G}\,:\,f\text{ is locally
  univalent}\} \, .$$
We will be mainly interested in the families
\begin{itemize}
\item[(i)]
$ \mathcal{G}=\hol(\Omega)$, and
\item[(ii)]
  $\mathcal{G}=\mathcal{\B}(\Omega)\coloneqq\left\{f\in\hol(\Omega)\,:\,\sup_{z\in\Omega} |f(z)|\leq 1\right\}
$.
\end{itemize}

\subsection{Runge--type theorems for locally univalent funcions}

For a compact set $K$ of $\C$ we denote by $\M_{l.u.}(K)$ the set of all
locally univalent meromorphic functions on (the components of) some open
neighborhood of $K$.

\begin{thm}[Runge--type theorem for locally univalent functions] \label{thm:r2}\label{thm:loc_univ_mer_runge}\label{thm:loc_univ_hol_runge}
Let $\Omega$ be a domain in $\C$ and let $K$ be a compact set in $\Omega$ such
that $\Omega \setminus K$ has no relatively compact components in
$\Omega$. Then
\begin{itemize}
\item[(a)] every $f \in \hol_{l.u.}(K)$ can be approximated uniformly on $K$ by
  functions in $\hol_{l.u.}(\Omega)$;
\item[(b)] every $f \in \M_{l.u.}(K)$ can be approximated $\chi$--uniformly on $K$ by
  functions in $\M_{l.u.}(\C)$, provided that $\C
  \setminus K$ is connected.
\end{itemize}
\end{thm}

Note that (b) is a somewhat weaker statement than (a). In fact, we do not 
 know if the analogue of (a) holds for meromorphic functions:

\begin{problem} \label{pro:r1}
Let $\Omega$ be a domain in $\C$ and let $K$ be a compact set in $\Omega$ such
that $\Omega \setminus K$ has no relatively compact components in
$\Omega$. Can every $f \in \M_{l.u.}(K)$ be approximated $\chi$--uniformly on $K$ by
  functions in $\M_{l.u.}(\Omega)$\,?
\end{problem}

The following example shows that we have to assume that $\Omega \setminus K$ has no relatively compact components in
$\Omega$ in Problem \ref{pro:r1}  and also that  $\C
  \setminus K$ is connected in Theorem \ref{thm:r2} (b). 
We employ the well--known fact that  a function $f\in\M(\Omega)$ is locally univalent if and only if its Schwarzian derivative
\begin{equation*}
	S_f\coloneqq \left(\frac{f^{\prime\prime}}{f^\prime}\right)^{\prime}-\frac{1}{2}\left(\frac{f^{\prime\prime}}{f^{\prime}}\right)^2
\end{equation*}
 is holomorphic on $\Omega$.

\begin{example}\label{ex:hol_convex}
Let $K:=\{z \in  \C \, :\, 1/2 \le |z| \le 2\}$ and 
$f(z):=-1/z^2 \in \M_{l.u.}(K)$. 	Suppose that there is a sequence
$(g_n)$ in $\mathcal{M}_{l.u.}(\C)$ which converges to $f$ $\chi$--uniformly on $K$.
	Then we have  $S_{g_n}\rightarrow S_f$ uniformly on $\partial\D$.
	But since $S_{g_n}\in\mathcal{H}(\C)$ for all $n\in\N$ the maximum principle implies $S_{g_n}\rightarrow h$ uniformly in $\D$ for a function $h\in\mathcal{H}(\D)$.
	We have $S_f\equiv h$ on $\D\cap K$ and hence on $\D\setminus\{0\}$.
	This, however, contradicts the fact that $0$ is a critical point of $f$.
\end{example}

Since Theorem \ref{thm:r2} (a) is a form of the classical Runge theorem in
which one allows only locally univalent functions,
it is tempting to ask if 
there are analogues of Mergelyan's approximation Theorem \cite{mergelyan1952uniform} and Arakelyan's Theorem \cite{arakelyan1968uniform} for locally univalent functions:
\begin{problem}\label{prob:lu_merg}
    Let $K$ be a compact set in $\C$  with connected complement.
    Suppose $f\colon K\rightarrow\C$ is continuous on $K$ and locally
    univalent in the interior $K^{\circ}$ of $K$.
    Can $f$ be approximated by entire locally univalent functions?    
    What if $K$ is only closed but unbounded and in addition $\hat{\C}\setminus K$ is locally connected at $\infty$?
\end{problem}
Note that we allow $f$ to have \enquote{critical points} on the boundary $\partial K$.
Recently, Andersson \cite{andersson2013mergelyan} has posed a similar problem
about  zero--free approximation.

\subsection{Universal locally univalent functions}

\begin{definition}
Let $\Omega$ be a domain in $\C$, $\mathcal{G} \subseteq \M(\Omega)$ and
$\Phi$ a family of holomorphic self--maps of $\Omega$.
 A function $G\in\mathcal{G}$ is called $\Phi$--universal in $\mathcal{G}$
 if 
$\{ G \circ \phi \, : \, \phi \in \Phi\}$ is dense in $\mathcal{G}$.
If $G\in\mathcal{G}$ is $\aut{\Omega}$--universal in  $\mathcal{G}$, we simply
call $G$ universal in $\mathcal{G}$.
\end{definition}

 Note that a $\Phi$--universal function in $\mathcal{G}$ is always supposed to
 belong to $\mathcal{G}$.

\medskip

The aim of this section is to provide necessary and also sufficient conditions
for the existence of $\Phi$--universal functions for families of 
locally univalent holomorphic or meromorphic functions on a domain $\Omega$ in $\C$.
For this purpose, the following concepts, which have  been introduced in   \cite{bernal1995universal}
and \cite{erdmann2009universal}, will play a crucial role.

\begin{definition}
	Let $\Omega$ be a domain in $\C$ and let $(\phi_n)$ be a sequence of holomorphic self--maps of $\Omega$.
	\begin{enumerate}[(i)]
		\item
		We say that $(\phi_n)$ is run--away, if for every compact set $K\subseteq \Omega$ there exists $n\in\N$ with $\phi_n(K)\cap K= \emptyset$.
		\item 
		We say that $(\phi_n)$ is eventually injective, if for every compact set $K\subseteq\Omega$ there exists $N\in\N$ such that the restriction $\phi_n|_K$ is injective for all $n\geq N$.
	\end{enumerate}
\end{definition}

These conditions turn out to be necessary for the existence of
$\Phi$--universal functions in $\hol_{l.u.}(\Omega)$:
\begin{prop}\label{prop:run-away}
Let $\Omega$ be a domain in $\C$ and let $\Phi$ be a family of locally
univalent self--maps of $\Omega$.
	Suppose that there is a $\Phi$--universal function in $\hol_{l.u.}(\Omega)$. 
	Then $\Phi$ contains a run--away and eventually injective sequence.
\end{prop}

We next turn to sufficient conditions, but 
restricting the discussion to the cases when $\Omega$ is either
simply connected or of infinite connectivity. The reason for this is the fact
that for domains of finite connectivity $N>1$  there are
$\Phi$--universal functions $f$ for $\hol(\Omega)$
such that the family $\Phi$ of locally univalent self--maps of $\Omega$ is mainly responsible for the denseness of
$\{f \circ \phi \, : \, \phi\in \Phi\}$ in $\hol(\Omega)$ and not $f \in
\hol(\Omega)$, see \cite{erdmann2009universal}.

\medskip
For simply connected domains, we have a complete picture:

\begin{thm}\label{thm:con_univ}
	Let $\Omega$ be a simply connected domain in $\C$. Suppose that
        $\Phi$ is a family of locally univalent self--maps of $\Omega$ which
        contains a run--away and eventually injective sequence.
	Then there is a $\Phi$--universal function in $\hol_{l.u.}(\Omega)$ and
 the set of all $\Phi$--universal
        functions in $\hol_{l.u.}(\Omega)$  is a dense $G_{\delta}$--subset of $\hol_{l.u.}(\Omega)$.
\end{thm}

In order to discuss the case of domains of infinite connectivity, we recall
that a compact subset $K$ of a domain $\Omega$ in $\C$ is called $\O$--convex if
$\Omega \setminus K$ has no relatively compact components in $\Omega$.

\begin{thm}\label{thm:mult_univ}
	Let $\Omega$ be a domain in $\C$ of infinite connectivity and let
        $\Phi$ be a family of locally univalent self--maps of $\Omega$.
	Suppose that there exists a sequence $(\phi_n)$ in $\Phi$ such that
	\begin{enumerate}[(i)]
		\item $(\phi_{n})$ is eventually injective, and
		\item for every $\O$--convex compact set $K$ in $\Omega$ and every $N\in\N$ there exists $n\geq N$ such that $\phi_{n}(K)\cap K=\emptyset$ and $\phi_{n}(K)\cup K$ is $\O$--convex.
	\end{enumerate}
Then there is a $\Phi$--universal function in $\hol_{l.u.}(\Omega)$ and the set
of all such functions  is a dense $G_{\delta}$--subset of $\hol_{l.u.}(\Omega)$.
\end{thm}

We now take a closer look at the case $\Phi \subseteq \aut{\Omega}$. 
It has been shown by Bernal-Gonz\'ales and Montes-Rodr\'iguez
\cite{bernal1995universal} that if $\Omega$ is
not conformally equivalent to $\C\setminus\{0\}$ then there is a $\Phi$--universal function
 in $\hol(\Omega)$ if and only if $\Phi$ contains a run--away sequence.
This result also holds in the setting of \textit{locally univalent} functions:

\hide{Theorems 3.1 and 3.6 in \cite{bernal1995universal} show that if $\Omega$ is
not conformally equivalent to $\C\setminus\{0\}$ and $(\phi_n)$ is a sequence
in $\aut{\Omega}$, then there is a $(\phi_n)$--universal function
 in $\hol(\Omega)$ if and only of $(\phi_n)$ is run--away.
The key step in the proof of the \enquote{only if part} is to show that such sequences preserve $\O$-convexity in the sense stated in assertion \textit{(ii)} of Theorem \ref{thm:mult_univ} (see \cite{bernal1995universal} Lemma 2.12).
Further note that the existence of a run--away sequence $(\phi_n)\subseteq\aut{\Omega}$ implies that either $\Omega$  simply connected or isomorphic to $\C\setminus\{0\}$ or of infinite connectivity.
This observation allows us to apply Theorem \ref{thm:mult_univ} and
Proposition \ref{prop:run-away} to obtain a corresponding result for locally univalent functions.}
\begin{thm}\label{cor:aut_univ}
	Let $\Omega$ be a  domain in $\C$ which is not conformally equivalent
        to $\C\setminus\{0\}$ and let $(\phi_n)\subseteq\aut{\Omega}$.
	Then there is a $(\phi_n)$--universal function in $\hol_{l.u.}(\Omega)$ if and only if $(\phi_n)$ is run--away.
\end{thm}

Note that Theorem \ref{thm:u1}  is a special instance of Theorem  \ref{cor:aut_univ}.

\begin{remark}
	If $\Omega$ is conformally equivalent to $\C\setminus\{0\}$ then there
        are no universal functions in $\hol_{l.u.}(\Omega)$.
	In fact, it was observed in \cite{bernal1995universal} that there
        are no   universal functions for  $\hol(\Omega)$ in this case.
	The argument is based on the maximum principle and stays valid for locally univalent functions.
\end{remark}

The final result of this section is concerned with universal locally univalent
\textit{meromorphic} functions.
Chan \cite{chan2001universal} has shown that there exists a meromorphic
function $f\in\M(\C)$ such that 
the set $T_f:=\{ f(\cdot+n) \, : \, n \in \N\}$ is dense in $\M(\Omega)$ for
every domain $\Omega\subseteq\C$. In the locally univalent
  situation we need to restrict the
discussion to simply connected domains since
the same reasoning as in Example \ref{ex:hol_convex} shows that if $T_f$ is
dense in $\M_{l.u.}(\Omega)$ for some
$f\in\M_{l.u.}(\C)$, then  $\Omega$ has to be simply connected.

\begin{thm}\label{thm:mer_loc_univ_birk}
    Let $\Omega\subseteq\C$ be a simply connected domain and let $\Phi$ be a
    family of locally univalent self--maps of $\Omega$ which contains a
    run--away and eventually injective sequence $(\phi_n)$. Then there is a $\Phi$--universal
    function in $\M_{l.u.}(\Omega)$ and the set of all such  functions is a dense $G_{\delta}$--subset of $\M_{l.u.}(\Omega)$.
\end{thm}

\medskip

 \begin{cor}  \label{cor} There exists a function $f\in\M_{l.u.}(\C)$ such that $T_f=\{ f(\cdot+n) \, : \, n \in \N\}$ is  dense in $\M_{l.u.}(\Omega)$ for every simply connected domain $\Omega$.
\end{cor}

\subsection{Universal bounded locally univalent functions}

Theorem \ref{cor:aut_univ}  says that if $\Omega$ is a simply connected domain in
$\C$ and $\Phi\subseteq \aut{\Omega}$, then there is a
$\Phi$--universal function in $\hol_{l.u.}(\Omega)$ if and only if
$\Phi$ contains a run--away sequence. The corresponding problem
for \textit{bounded} locally univalent functions is the following:

\begin{problem} \label{pro:r2}
Let  $\Omega\not=\C$ be a simply connected domain in
$\C$ and let $(\phi_n)$ be a run--away sequence in $\aut{\Omega}$. Does there
exist a $(\phi_n)$--universal function in $\B_{l.u.}(\Omega)$\,?
\end{problem}

We can only prove the weaker result (see Theorem \ref{thm:u2}) that
for any simply connected domain $\Omega \subsetneq \C$ there is an
universal function in $\B_{l.u.}(\Omega)$. The idea is to ``remove'' the
critical points of  a universal function in $\B(\Omega)$ while keeping the
function bounded. Such a construction is based on the following
considerations for which we initially assume that $\Omega$ is the unit disk $\D$.

\medskip
For any non--constant $f\in\M(\D)$ we  denote by $\Omega_f$  the set of all non-critical
points of $f$. Then $\Omega_f$ is a subdomain of $\D$, so 
by the Uniformization Theorem there exists a universal covering map $\Psi$
from $\D$ onto $\Omega_f$ which is uniquely determined up to pre-composition
with a unit disc automorphism. Clearly,  $f\circ\Psi$ is a locally univalent
meromorphic function on $\D$.

\begin{thm}\label{thm:lu_universal_functions}
	Let $\mathcal{G}\subseteq\M(\D)$ and let  $G\in\M(\D)$ be a
        non-constant universal function in $\mathcal{G}$. Suppose that $\Psi$
        is a universal covering map from $\D$ onto $\Omega_G$.
	Then $F\coloneqq G\circ\Psi \in \mathcal{M}_{l.u.}(\D)$ and for each 
$f \in \mathcal{G}_{l.u.}$ there is a sequence $(\phi_n)$ in $\aut{\D}$ such that
$F \circ \phi_n \to f$ locally uniformly in $\D$ w.r.t.~the chordal metric.
\end{thm}

Note that Theorem \ref{thm:lu_universal_functions} does not assert that
$F$ is $\Phi$--universal in $\mathcal{G}_{l.u.}$, since in general
$F \not\in \mathcal{G}$. 

\medskip
 
Theorem \ref{thm:u2} for $\Omega=\D$ follows immediately from Theorem
\ref{thm:lu_universal_functions} by taking for $G$ a universal function in
$\B(\D)$. One can even take a universal Blaschke product in
$\B(\D)$,  see, for instance, \cite{heins1954universal} and \cite{gorkin2004universal}.
Using the Riemann Mapping Theorem, the general case of Theorem \ref{thm:u2} that $\Omega$ is a proper
simply connected subdomain of $\C$ easily reduces to the case of the unit disk.

\subsection{Universal conformal metrics with constant curvature}

In this section we investigate universal conformal metrics with constant
curvature. We shall use  the following terminology.
\begin{definition}
Let $\Omega$ be a domain in $\C$ and let $\Phi$ be a family of locally univalent
self--maps of $\Omega$. Let
   $c$ be a real number. We call   $\Lambda \in \Lambda_c(\Omega)$
   $\Phi$--universal in $\Lambda_c(\Omega)$ if the family of pullbacks
$\{\phi^*\Lambda \, : \, \phi \in \Phi\}$ is dense in $\Lambda_c$.
We call $\Lambda$ universal in $\Lambda_c(\Omega)$  if $\Lambda$
is $\aut{\Omega}$--universal in $\Lambda_c(\Omega)$.
\end{definition}

We note that for $c<0$ the set $\Lambda_c(\Omega)$ is non--empty if and only
if $\Omega$ is a hyperbolic domain, that is, $\C \setminus \Omega$ has at
least two distinct points.

\begin{thm} \label{thm:u4}
Let $\Omega$ be a simply connected domain in $\C$ and let $c$ be a
non--negative real number. Suppose that 
 $\Phi$ is a family of locally univalent self--maps of $\Omega$ which contains
 a run--away and eventually injective sequence.
 Then there exists a $\Phi$--universal $\Lambda \in \Lambda_c(\Omega)$.
\end{thm}

In particular, under the assumptions of Theorem \ref{thm:u4} there
is always an $\aut{\Omega}$--universal
element in $\Lambda_c(\Omega)$. This proves the cases $c \ge 0$ of Theorem \ref{thm:u3}.
The cases of constant \textit{negative} curvature appear to be more difficult:
 
\begin{problem} \label{prop:metric}
Let $\Omega$ be a simply connected proper subdomain of $\C$ and $c<0$.
Which families $\Phi$ of locally univalent self--maps of $\Omega$ 
admit $\Phi$--universal elements in $\Lambda_c(\Omega)$?
\end{problem}

Only in the  case $\Phi=\aut{\Omega}$ we are able to show that the answer to
Problem \ref{prop:metric} is affirmative. This, at least,  proves  the cases $c <0$ of
Theorem \ref{thm:u3}.

\medskip

The proof of Theorem \ref{thm:u4} and the case $\Phi=\aut{\Omega}$ of Problem
\ref{prop:metric} below is  based on the classical
representation theorem of  Liouville for constantly curved conformal metrics
in terms of  locally univalent holomorphic maps and the universality results
for these maps from Section 2.2 and Section 2.3.
A major drawback of this approach lies in the fact that Liouville's
representation theorem only works for \textit{simply connected} domains.



\medskip

The case of constant curvature $0$ is particularly interesting since
a conformal metric $\lambda(z) \, |dz|$ has curvature $0$ if and only if $u:=
\log \lambda$ is harmonic. We can therefore also employ the Runge--type
theorems for harmonic functions (see for example Theorem 3 in
\cite{gauthier1980uniform} or Theorem 4 in \cite{gardiner1994superharmonic})
to prove existence of universal conformal metrics with constant curvature $0$,
even for not simply connected domains:

\begin{thm}\label{thm:harm_univ}	Let $\Omega\subseteq\C$ be a domain of infinite connectivity and let
        $\Phi$ be a family of locally univalent self--maps of $\Omega$. 
	Suppose that there exists a sequence $(\phi_n)$ in $\Phi$ such that
	\begin{enumerate}[(i)]
		\item $(\phi_{n})$ is eventually injective, and
		\item for every $\O$--convex compact set $K\subseteq\Omega$ and every $N\in\N$ there exists $n\geq N$ such that $\phi_{n}(K)\cap K=\emptyset$ and $\phi_{n}(K)\cup K$ is $\O$--convex.
	\end{enumerate}
Then the following hold:
    \begin{enumerate}[(a)]
        \item
            There exists a harmonic function $u$ on $\Omega$ such that
$\{u \circ \phi \, : \, \phi \in \Phi\}$ is dense in the set of all harmonic
functions on $\Omega$.
        \item There exists a $\Phi$--universal element in $\Lambda_0(\Omega)$.
    \end{enumerate}
\end{thm}

For other universality results for harmonic functions, see
\cite{dzagnidze1964universal}, \cite{armitage2003harmonic} or \cite{calderonmueller2004}.

\begin{remark}
Suppose that $u$ is a harmonic function on a domain $\Omega$ in $\C$ such that the conclusion
(a) of Theorem \ref{thm:harm_univ} holds. Then one might be
inclined to suspect that $e^{u}$ is an $\Phi$--universal element in
$\Lambda_0(\Omega)$. However, the presence of the 
``conformal factor'' $|\phi'(z)|$ in the pullback
$\phi^*e^{u}(z)=e^{u(\phi(z))} |\phi'(z)|$ calls this into question.
\end{remark}

\begin{remark}[SK--metrics]
Finally we would like to indicate that it is tempting to replace the classes
$\Lambda_c(\Omega)$ of conformal metrics with constant curvature $c$ by other
classes of conformal metrics which are invariant under pullback.
The most prominent example is perhaps the class $\text{SK}(\Omega)$ of
(densities of) SK--metrics
 as introduced by M.~Heins \cite{heins1962class}. Now, in order to study
 universality aspects for $\text{SK}(\Omega)$  one  first needs to specify the underlying
 topology. In view of the fact that (densities of) SK--metrics are subharmonic
 functions, there  are two natural topologies for this purpose (see \cite{gauthier2007universal}):
\begin{itemize}
\item[(T1)] Topology of locally uniform convergence; 
\item[(T2)] Topology of decreasing convergence.
\end{itemize}
Both topologies have been employed e.g.~in \cite{gauthier2007universal} for
proving universality results for  subharmonic functions. While the (T1) topology is
naturally restricted to the subclass of \textit{continuous} subharmonic
functions (see the discussion in \cite{gauthier2007universal}, in particular Section 3), the
(T2) topology  turns out to be fairly natural for the class of \textit{all} subharmonic functions (see e.g.~Theorem
3.4 in \cite{gauthier2007universal}). However, the (T2) topology is not convenient
for studying universality for SK--metrics\,! The reason is that one can show 
that the strong form of  Ahlfors' Lemma (\cite[Theorem 7.1]{heins1962class})
implies that w.r.t.~the topology of decreasing convergence the
density of the hyperbolic metric
$\lambda_{\Omega}(z)\dz$ of the domain $\Omega$ is an \textit{isolated} point of $\text{SK}(\Omega)$ and 
in fact the  only candidate for a universal SK--metric in this
setting. This, however,  is clearly not possible
since $\phi^*\lambda_{\Omega}=\lambda_{\Omega}$ for every $\phi \in \aut{\Omega}$.
\end{remark}

We are therefore led to consider the set  $\text{SK}_c(\Omega):=\text{SK}(\Omega) \cap
C(\Omega)$ of all \textit{continuous} SK--metrics equipped the topology of
locally uniform convergence.

\begin{problem}
Let $\Omega \subseteq \C$ be a hyperbolic domain.
    Does there exist a continuous \textrm{SK}--metric
    $\Lambda(z)\dz$ on $\Omega$  such that
$ \{ \phi^* \Lambda \, : \, \phi \in \aut{\Omega}\}$ is dense in $\textrm{SK}_c(\Omega)$\,?
\end{problem}

\section{Proofs}

We first introduce some notation.

\medskip

We denote by  $\operatorname{tr}(\gamma)$  the
trace of a curve $\gamma$ in $\C$ and by $\operatorname{ind}_{\gamma}(z)$  the winding number of $\gamma$ around $z$.
Let $U$ be an open set in $\C$ and let $\hol_{\neq0}(U)$ be the set of
all functions in $\hol(U)$ with no zeros in $U$.
For a set $M$ in $\C$ we write $f\in\hol_{\neq0}(M)$ if there is
an open neighborhood $U$ of $M$ such that $f\in\hol_{\neq0}(U)$.
For a compact set $K$ in $\C$ we  set
$$ ||f-g||_K:=\max \limits_{z \in K} |f(z)-g(z)| $$
if $f$ and $g$ are holomorphic functions in a neighborhood of $K$, 
and
$$ \chi_K(f,g):=\max \limits_{z \in K} \chi(f(z),g(z)) $$
if $f$ and $g$ are meromorphic functions in a neighborhood of $K$.

\begin{prop}\label{prop:K_connected_Oconvex1}
	Let $\Omega$ be a domain in $\C$, $K$ a compact $\O$--convex set in
        $\Omega$  and $\varepsilon>0$.
\begin{enumerate}[(a)]
\item 
	Suppose $f\in\hol_{\neq0}(K)$.
	Then there exists a connected compact $\O$--convex set $M$ in $\Omega$
        with piecewise differentiable boundary $\partial M$ such that $K\subseteq M$ and a function $g\in\hol_{\neq0}(M)$ with $\|f-g\|_K<\varepsilon$.
\item 
	Suppose $f\in\M_{l.u.}(K)$.
	Then there exists a compact $\O$--convex set $M$ in $\Omega$ with
        connected interior $M^{\circ}$  such that $K\subseteq M^{\circ}$ and a
        function $g\in\M_{l.u.}(M)$ with $\chi_K(f,g)<\varepsilon$.
If $f \in\hol_{l.u.}(K)$, then $g \in\hol_{l.u.}(M)$ with $\|f-g\|_K<\varepsilon$. 
\end{enumerate}
\end{prop}

\begin{proof}
	We only prove part \textit{(a)}; the proof of part \textit{(b)} is similar.
	By the classical theorem of Runge and by Hurwitz' theorem, there
        exists a rational function $g\in\hol(\Omega)\cap\hol_{\neq0}(K)$ such that $\|f-g\|_K<\varepsilon$.
	Let $z_1,\dots,z_N$ be the zeros of $g$ in $\Omega$.
	Since $K$ is $\O$--convex, there exist paths $\gamma_j\colon[0,1)\rightarrow\Omega\setminus K$ with $\gamma_j(0)=z_j$, $\gamma_j(t)\rightarrow\partial\Omega$ for $t\rightarrow 1$ and such that $W\coloneqq \Omega\setminus(\operatorname{tr}(\gamma_1)\cup\cdots\cup\operatorname{tr}(\gamma_N))$ is connected.
	Note that $W$ is open and $K\subseteq W$.
	Let $(M_n)$ be a compact exhaustion of $W$ with connected compact
        $\O$--convex sets in $W$ such that $\partial M_n$ is piecewise
        differentiable for each $n\in\N$. Since  a compact set in $W$ is
        $\O$--convex in $W$ if and only if it is $\O$--convex in $\Omega$, we 
can take $M=M_n$ with $n\in\N$ sufficiently large so that $M=M_n 
\supseteq K$.
\end{proof}

\subsection{Proof of Theorem \ref{thm:loc_univ_hol_runge} (a)}

The  proof of Theorem \ref{thm:loc_univ_hol_runge} is based on the following theorem:
\begin{thm}\label{thm:deriv_runge}
	Let $\Omega$ be a domain in $\C$, $K$ a $\O$--convex compact set in
        $\Omega$ and $g\in\hol_{\neq0}(K)$.
	Then there exists a sequence $(f_m)\subseteq\hol_{\neq0}(\Omega)$ such that $\limm f_m=g$ uniformly on $K$ and
	\begin{equation}\label{eq:int_cond1}
		\int_{\gamma} f_m(z)dz = \int_{\gamma} g(z)dz
	\end{equation}
	for every closed curve $\gamma\subseteq K$ and every $m\in\N$.
\end{thm}
\begin{remark}
	Theorem \ref{thm:deriv_runge} is in fact a special case of  Theorem 2 in \cite{majcen2007closed}, where closed holomorphic 1-forms on Stein manifolds are considered.
We give a direct and simpler proof for the one--dimensional situation of
Theorem \ref{thm:deriv_runge}.
\end{remark}
\begin{proof}
	By Proposition \ref{prop:K_connected_Oconvex1}\textit{(a)}
         we may assume that $K$ is connected and $\partial K$ is piecewise differentiable.
	Let $D_1,\dots,D_n$ be the bounded connected components of $\C\setminus K$.
	For $j=1,\dots,n$ choose  $z_j\in D_j\setminus\Omega$ and 
let $\gamma_j$ be a parametrization of the positively oriented boundary $\partial D_j$.
	Then $\operatorname{ind}_{\gamma_k}(z_j)=\delta_{kj}$.
	The connectedness of $K$ implies that
        $\Gamma\coloneqq\bigcup_{k=1}^n\operatorname{tr}(\gamma_k)$ is
        a compact $\O$--convex set in $\Omega$. 
	Since every closed curve in $K$ is homologous to a linear combination
        of the curves $\gamma_1,\ldots, \gamma_n$ with integer coefficients, it suffices  to find a
        sequence $(f_m)\in\hol_{\neq0}(\Omega)$ such that $\limm f_m =g$
        uniformly on $K$ and equation (\ref{eq:int_cond1}) holds for 
        $\gamma =\gamma_k$ for every $k=1,\ldots, n$.
	
\medskip

Now for any $j=1,\dots,n$, Runge's Theorem implies that there is a sequence
$(w_{j,m})_m$ in $\hol(\Omega)$ with 
	\begin{equation*}
		\limm w_{j,m}(z)=\frac{1}{g(z)(z-z_j)}
	\end{equation*}
	uniformly on $\Gamma$.
	In particular, 
	\begin{equation*}
		\limm\left(\int_{\gamma_k} w_{j,m}(z)g(z)dz\right)_{k,j=1,\dots,n} = E_n,
	\end{equation*}
	where $E_n\in\C^{n\times n}$ is the identity matrix.
	Hence we can find a $\mu\in\N$ such that the matrix
	\begin{equation*}
		A\coloneqq \left(\int_{\gamma_k} w_{j,\mu}g(z)dz\right)_{k,j=1,\dots,n}
	\end{equation*}
	is non--singular.

\medskip 	
	By a well--known extension of Runge's Theorem (Theorem 6.3.1 in
        \cite{narasimhan1985complex}) there exists a sequence $(g_m)$ in $\hol_{\neq0}(\Omega)$ such that $\limm g_m=g$ uniformly on $K$.
	Consider the functions
	\begin{eqnarray}
		\psi_{k}\colon\C^n\rightarrow\C,   	& &  (s_1,\dots,s_n)\mapsto\int_{\gamma_k}\operatorname{exp}\left(\sum_{j=1}^n s_j w_{j,\mu}(z)\right)g(z)\,dz,\nonumber\\
		\psi_{k,m}\colon\C^n\rightarrow\C,	& &  (s_1,\dots,s_n)\mapsto\int_{\gamma_k}\operatorname{exp}\left(\sum_{j=1}^n s_j w_{j,\mu}(z)\right)g_m(z)\,dz,\nonumber
	\end{eqnarray}
	and the entire functions $\psi,\psi_m\colon\C^n\rightarrow\C^n$ defined by $\psi(s)\coloneqq(\psi_{1}(s),\dots,\psi_n(s))$ and $\psi_m(s)\coloneqq(\psi_{1,m}(s),\dots,\psi_{n,m}(s))$.
	Then $\limm\psi_m=\psi$ locally uniformly on $\C^n$ and $D\psi(0) = A$ is non--singular.
	Hence there exists a sequence $(s_m)=(s_{1,m},\dots,s_{n,m})$ in $\C^n$ with $\limm s_m = 0$ and $\psi_m(s_m)=\psi(0)$.
	This concludes the proof with
	\begin{equation*}
	f_m(z) = \operatorname{exp}\left(\sum_{j=1}^n s_{j,m} w_{j,\mu}(z)\right)g_m(z).
	\end{equation*}
\end{proof}
\begin{proof}[Proof of Theorem
  \ref{thm:loc_univ_hol_runge}(a)] 
	By Proposition\ref{prop:K_connected_Oconvex1}\textit{(b)} we may
        assume that $f \in \hol_{l.u.}(M)$ for some  connected $\O$--convex compact set $M$
        of $\Omega$ with smooth boundary and such that $K\subseteq M^{\circ}$. Hence we can apply 
       Theorem \ref{thm:deriv_runge} to $f'\in \hol_{\not=0}(M)$, so there exists a sequence $(g_n)\subseteq\hol_{\neq0}(\Omega)$ with $\limn g_n =f^{\prime}$ uniformly on $M$ and
	\begin{equation*}
	\int_{\gamma} g_n(z)\, dz = \int_{\gamma} f^{\prime}(z)\,dz = 0
	\end{equation*}
	for every closed curve $\gamma$ in $M$.
            
\medskip 
	Now we choose a compact exhaustion
        $(K_k)_{k}$  of $\Omega$ by connected
        $\O$--convex sets in $\Omega$ with smooth boundaries and such that $K_1=M$.
     Suppose we have fixed arbitrary numbers $\varepsilon>0,\ k\in\N$ and a function $h\in\hol_{\neq0}(\Omega)$ with $\int_{\gamma} h(z)dz = 0$ for every closed curve  $\gamma$ in $K_k$.
         Then by \cite[Lemma 4]{gunning1967immersion} there exists a function $v\in \hol(\Omega)$ with $\|v\|_{K_k}<\varepsilon$ and $\int_{\gamma} e^{v(z)} h(z) dz = 0$ for every closed curve $\gamma$ in $K_{k+1}$.
	 From this fact and an obvious induction argument, we can deduce that
         there exists a sequence $(v_{n,k})_k$ in $\hol(\Omega)$ with $\|v_{n,k}\|_{K_k} < \frac{1}{2^k n}$ and such that for every closed curve $\gamma$ in $K_k$ we have
    \begin{equation*}
       	\int_{\gamma}\exp\left(\sum_{j=1}^{k} v_{n,j}(z)\right) g_n(z)\, dz = 0.
    \end{equation*}
    We define a holomorphic function $w_n\in\hol(\Omega)$ by
    \begin{equation*}
        w_n(z)\coloneqq\sum_{j=1}^\infty v_{n,j}(z).
    \end{equation*}
     Clearly we have $\|w_n\|_K<\frac{1}{n}$ and 
	\begin{equation*}
		\int_{\gamma}e^{w_n(z)}g_n(z)dz = 0
	\end{equation*}
	for every closed curve $\gamma$ in $\Omega$.
	This means that for fixed $z_0\in K$ and for
	 each $n$ there is an anti--derivative $G_n \in \hol(\Omega)$ of
         $e^{w_n} g_n$ with $G_n(z_0)=f(z_0)$. By construction, $G_n\in\hol_{l.u.}(\Omega)$.
	Since $M$ is connected and $\limn G_n^{\prime}=f^{\prime}$ uniformly
        on $M$ we conclude $\limn G_n=f$ uniformly on $M$ and hence on $K$.
\end{proof}

\subsection{Proof of Theorem \ref{thm:loc_univ_mer_runge} (b)}
%

By Proposition \ref{prop:K_connected_Oconvex1} \textit{(b)} we may assume that
$f \in \M_{l.u.}(M)$ for some compact $\O$--convex set $M$ in $\C$ whose interior
$G:=M^{\circ}$ is  connected and contains $K$. 
Since $f$ is locally univalent in a neighborhood of $M$, its Schwarzian derivative $S_f$
is holomorphic there, so $S_f \in \hol(M)$. According to some basic facts about
complex differential equations, see e.g.~\cite[Theorem 6.1]{laine1993nevanlinna}, we can
recover $f$ from $S_f$ by writing $f$ as the quotient, 
$$ f=\frac{u_1}{u_2} \, ,$$
of two linearly independent solutions $u_1,u_2\in\hol(\Omega)$ of
the homogeneous linear differential equation
\begin{equation}\label{eq:con_schwarzian}
w^{\prime\prime} +\frac{1}{2}S_f\cdot w = 0 \, .
\end{equation}
Since $S_f \in \hol(M)$ and $\C \setminus M$ has no bounded components, the
classical Runge theorem shows that  there exist polynomials
$p_n\colon\C\rightarrow\C$ such that
\begin{equation*}
 p_n  \to S_f\qquad\text{uniformly on }M.
\end{equation*}
We now consider  the homogeneous linear differential equations corresponding to
the polynomials $p_n$.
Fix $z_0\in G$ with $u_2(z_0)\neq 0$,
and let $v_n\in \hol(\C)$ be the unique solution of the initial value problem
\begin{equation*}
v_n^{\prime\prime} + \frac{1}{2}p_n\cdot v_n= 0,\qquad
v_n(z_0)=u_1(z_0),\quad v_n^{\prime}(z_0) =u_1^{\prime}(z_0) \, .
\end{equation*}
Then we clearly have
\begin{equation*}
v_n(z) = u_1(z_0)+ u_1^{\prime}(z_0)(z-z_0)-\frac{1}{2}\int_{z_0}^{z}\,(z-\xi)
p_n(\xi)v_n(\xi)\,d\xi \, , \quad z \in \C \,.
\end{equation*}
Hence a standard application of Gronwall's lemma
(cf.~\cite[Lemma 5.10]{laine1993nevanlinna})
  shows that the sequence
$(v_n)$ is locally bounded in $G$. We are therefore in a position to apply
 Montel's theorem and conclude that $\{v_n \, : \, n \in \N\}$ is a normal
 family. Clearly, every subsequential limit function  $v \in \hol(G)$ of $(v_n)$ is a
 solution of  (\ref{eq:con_schwarzian}) with $v(z_0)=u_1(z_0)$ and
 $v^{\prime}(z_0)=u_1^{\prime}(z_0)$ in $G$. By uniqueness of this solution,
 we conclude $v=u_1$. Consequently, we have
$$ v_n \to  u_1  \quad \text{ locally uniformly in } G \, .$$
For the unique solution $w_n \in \hol(\C)$ of the initial value problem
\begin{equation*}
w_n^{\prime\prime} + \frac{1}{2}p_n\cdot w_n =
0,\qquad 
w_n(z_0)=u_2(z_0),\quad w_n^{\prime}(z_0) = u_2^{\prime}(z_0) \, ,
\end{equation*}
we arrive in a similar way at 
$$ w_n \to  u_2 \quad \text{ locally uniformly in } G \, . $$
We claim that $v_n$ and $w_n$ are linearly independent for large $n$.
For this purpose we consider the Wronskian
$$ W(h,g)=h g'- h' g \in \hol(G) \quad \text{ for } g,h \in \hol(G) \, .$$
 Since $u_1$ and $u_2$ are  solutions of the linear differential equation
 (\ref{eq:con_schwarzian}), there is a constant $\lambda \in \C$ such that
 $W(u_1,u_2)(z)=\lambda$ for all $z \in G$, see
 \cite[Proposition 1.4.8]{laine1993nevanlinna}. In a similar way, we see that
for each $n \in \N$ there is $\lambda_n \in  \C$ such that
$W(v_n,w_n)(z)=\lambda_n$ for all $z \in \C$. By what we have already shown,
$\lambda_n \to \lambda$ as $n \to \infty$. Since $u_1$ and $u_2$ are
linearly independent, we have $\lambda\not=0$, see \cite[Proposition
1.4.2]{laine1993nevanlinna}. Hence $\lambda_n\not=0$, so $v_n$ and $w_n$ are
linearly independent for all but finitely many
$n \in \N$. 

\medskip

We can therefore apply Theorem 6.1 in \cite{laine1993nevanlinna} which implies
that
$$ g_n:=\frac{v_n}{w_n} \in \M_{l.u.}(\C) \, .$$
Since $v_n \to u_1$ and $w_n \to u_2$ locally uniformly in $G$, we see that
$g_n \to u_1/u_2=f$ locally uniformly in $G$ w.r.t.~the chordal metric, so in
particular $\chi$--uniformly on $K$.\vspace*{0mm}\hfill\qedsymbol

\subsection{Proof of Proposition \ref{prop:run-away}}
	Suppose that $u\in\hol_{l.u.}(\Omega)$ is $\Phi$--universal in
        $\hol_{l.u.}(\Omega)$, and let 
        $K$ be a compact set in $\Omega$. Choose  a compact set 
	$L$ in $\Omega$ which contains $K$ in its interior and which is the
        closure of its interior.
	Let 
$$\delta\coloneqq \frac{1}{2}\dist(K,\partial L)>0 \, , \quad M:=\sup_{z\in L} |z| \, .$$ 
Then  $f(z)\coloneqq z+2M+2\delta$ belongs to $\hol_{l.u.}(\Omega)$.
	Since $u$ is $\Phi$--universal in $\hol_{l.u.}(\Omega)$ there exists $\phi\in \Phi$ such that
	\begin{equation}\label{eq:un1}
		\|u\circ\phi-f\|_L<\delta.
	\end{equation}
	This in particular implies $|u(\phi(z))|\geq |f(z)|-\delta \geq
        M+\delta$ for all $z\in K$, so
$$ \min \limits_{z \in K} |u(\phi(z))|\ge M+\delta >M \ge \max \limits_{\phi(z) \in
  K} |\phi(z)| \, ,$$
and thus $\phi(K)\cap K=\emptyset$. Next, we fix $z_0 \in K$.  
	Then the estimate  (\ref{eq:un1})  shows that for every
        $z\in\partial L$ we have
	\begin{equation*}
		\big|[u(\phi(z_0))-u(\phi(z))]-[z_0-z]\big|<2 \delta
                \le|z_0-z| \,. 
	\end{equation*}
	Hence, by Rouch\'e's theorem, $u(\phi(z_0))-u(\phi(z))$ and $z_0-z$ have
        the same numbers of zeros in $L^{\circ}$. This implies that $\phi$ is
        injective on $K$.

\medskip
	Finally let $(K_n)$ be a compact exhaustion of $\Omega$.
	We can apply the reasoning above to each $K_n$ to obtain a run--away
        and eventually injective sequence in $\Phi$.

\subsection{Proofs of Theorems \ref{thm:con_univ}  and  
  \ref{thm:mer_loc_univ_birk}, and Corollary \ref{cor}} 

We are going to apply a fairly standard universality criterion.
Let $\mathcal{T}$ be a collection of
continuous self--maps of a topological space $X$. We say that $\mathcal{T}$
acts transitively on $X$ if for every pair of open sets $U$ and $V$ in $X$ there is
an $\tau \in\mathcal{T}$ such that $\tau(U) \cap V\not=\emptyset$.
An element $u\in X$ is called universal for $\mathcal{T}$ if the orbit $\{\tau
(u)\,:\,\tau\in\mathcal{T}\}$ is dense in $X$. 
\begin{thm}[Birkhoff transitivity criterion]\label{thm:univ_crit}
	Let $X$ be a second countable Baire--space and $\mathcal{T}$ a family
        of continuous self--maps of $X$. Suppose that $\mathcal{T}$ acts
        transitively on $X$. Then there exists an universal element for
        $\mathcal{T}$ and the set of all universal elements for $\mathcal{T}$ is a dense $G_{\delta}$--subset of $X$.
\end{thm}
For a proof see for instance  \cite[Theorem 1]{erdmann1999universal}.
For later use, we note that if $X$ is a separable metric space, then a
collection $\mathcal{T}$ of continuous self--maps of $X$  acts 
transitively on $X$ if and only if for every pair of points $v$ and $w$ in $X$ there exist a sequence $(v_n)$ in $X$ and a sequence $(\tau_n)$ in $\mathcal{T}$ such that $v_n\rightarrow v$ and $\tau_{n}(v_n)\rightarrow w$.

\medskip

We now apply these concepts to investigate universality for holomorphic and
meromorphic functions.
Let $\Omega$ be a domain in $\C$. We associate to any  holomorphic
self--map $\phi$ of $\Omega$ the composition operator
$$ C_{\phi}\colon\hol(\Omega)\rightarrow\hol(\Omega)\, ,\quad f\mapsto f\circ\phi \, .$$ 
If $\phi$ is locally univalent, then $C_{\phi}$ maps $\hol_{l.u.}(\Omega)$ into $\hol_{l.u.}(\Omega)$.
Since the union of $\hol_{l.u.}(\Omega)$ with all constant functions is a complete metric space and $\hol_{l.u.}(\Omega)$ is an open subset of this space, $\hol_{l.u.}(\Omega)$ is a Baire--space.
\begin{proof}[Proof of Theorem \ref{thm:con_univ}:]
In view of Theorem  \ref{thm:univ_crit} it suffices to show that
 the family $\{C_{\phi} \, : \, \phi \in \Phi\}$ acts transitively on $\hol_{l.u.}(\Omega)$.
	Let $f,g\in\hol_{l.u.}(\Omega)$. Since $\Omega$ is simply connected
        there is an exhaustion $(K_n)$ of
        $\Omega$ with compact sets $K_n$ in $\Omega$ such that each $K_n$ has
        connected complement. By assumption there is a sequence $(\phi_n)$ in
        $\Phi$ such that  $\phi_{n}$ is injective on $K_n$ and
        $\phi_{n}(K_n)\cap K_n=\emptyset$ for each $n \in \N$.
	Define $L_n\coloneqq K_n\cup\phi_{n}(K_n)$ and $h_n\in\hol_{l.u.}(L_n)$ by
	\begin{equation}
		h_n(z)\coloneqq
	\begin{cases}
		f(z),& z\in K_n\\
		g(\phi_{n}^{-1}(z)),&z\in \phi_{n}(K_n).
	\end{cases}
	\end{equation}
	Note that each $L_n$ has connected complement, hence by Theorem
        \ref{thm:loc_univ_hol_runge} (a) there exists a function $f_n\in\hol_{l.u.}(\Omega)$ with $\|f_n-h_n\|_{K_n}\leq1/n$.
	This implies $f_n\to f$ and $f_n\circ\phi_{n}\to g$ locally uniformly in $\Omega$.	
\end{proof} 	

The proof of Theorem \ref{thm:mer_loc_univ_birk} is identical except for that
we need to apply part (b) of Theorem \ref{thm:loc_univ_hol_runge} instead of
part (a).

\begin{proof}[Proof of Corollary \ref{cor}]
    By   Theorem \ref{thm:mer_loc_univ_birk} there exists a universal function $f\in\M_{l.u.}(\C)$ such that $T_f$ is dense in $\M_{l.u.}(\C)$.
        Let $\Omega\subseteq\C$ be a simply connected domain.
        Then, as a consequence of Theorem \ref{thm:loc_univ_hol_runge} (b), $\M_{l.u.}(\C)$ is  dense in $\M_{l.u.}(\Omega)$.
        This fact together with the universality of $f$ implies, that $T_f$ is dense in $\M_{l.u.}(\Omega)$.
\end{proof}

\subsection{Proofs of Theorems   \ref{thm:mult_univ},
   \ref{cor:aut_univ}  and \ref{thm:harm_univ}}

Except for  Theorem \ref{thm:harm_univ} (b), 
the proofs are  similar to the proof of Theorem \ref{thm:con_univ}, so we
only indicate the modifications that are necessary.

\begin{proof}[Proof of Theorem \ref{thm:mult_univ}]
We start with an  exhaustion $(K_n)$ of
        $\Omega$ with compact sets $K_n$ in $\Omega$. We can assume that each
        $K_n$ is $\O$--convex in $\Omega$. By hypothesis, there is a sequence
 $(\phi_n)$ in $\Phi$ that has all the properties we need in order to proceed as in the
 proof of Theorem \ref{thm:con_univ}.
\end{proof}

\begin{proof}[Proof of Theorem \ref{cor:aut_univ}]
By Proposition \ref{prop:run-away} we only need to show the ``if''--part.
 Since by hypothesis, 
$\Omega$ is not conformally equivalent to $\C\setminus \{0\}$,
the existence of a run--away sequence in $\aut{\Omega}$ implies that $\Omega$ is either
simply connected  or of infinite connectivity, see the
discussion following Lemma 2.9 in \cite[p.~51--52]{bernal1995universal}.
 In the first case, when $\Omega$ is simply connected, we
can simply apply Theorem \ref{thm:con_univ}. In the second case, when $\Omega$ is of infinite connectivity, we start with an  exhaustion $(K_n)$ of
        $\Omega$ with $\O$--convex compact sets $K_n$ in $\Omega$.
        Since $\phi_n$ is run-away we may assume that
        $\phi_n(K_n) \cap K_n=\emptyset$ for each $n \in \N$. Then, by a key
        observation (see \cite[Lemma 2.12]{bernal1995universal}), 
        it follows that $\phi_n(K) \cup K$ is $\O$--convex in $\Omega$. We are
        thus in a position to apply Theorem \ref{thm:mult_univ}.
\end{proof}

\begin{proof}[Proof of Theorem \ref{thm:harm_univ} (a)]
The proof is identical to the proof of Theorem  \ref{thm:mult_univ} except
that one has to apply a Runge-type theorem for harmonic functions (see for
example Theorem 3 in \cite{gauthier1980uniform} or Theorem 4 in
\cite{gardiner1994superharmonic}) instead of Theorem
\ref{thm:loc_univ_hol_runge} (a).
\end{proof}

\begin{proof}[Proof of Theorem \ref{thm:harm_univ} (b)]
    We show that the collection of continuous maps
$$ \Lambda_0(\Omega) \to \Lambda_0(\Omega) \, , \quad \lambda\mapsto
\phi^*\lambda \, , \qquad \phi \in \Phi \, ,$$
acts transitively on $\Lambda_0(\Omega)$ and then apply Theorem \ref{thm:univ_crit}.
    Let $(K_n)$ be a compact exhaustion of $\Omega$ such that each $K_n$ is $\O$--convex.
    Then by the assumptions on $\Phi$ we can find a sequence $(\phi_n)$ in
    $\Phi$ such that each $\phi_n$ is injective in an open neighborhood of $K_n$, $K_n\cap \phi_{n}(K_n)=\emptyset$ and $L_n\coloneqq K_n\cup\phi_{n}(K_n)$ is $\O$--convex.
    Let $\lambda, \mu \in\Lambda_0(\Omega)$ and define $h_n\colon L_n\rightarrow\R$ by
    \begin{equation*}
        h_n(z)\coloneqq\begin{cases} \log\lambda(z) \, ,& z\in K_n\\ \log
          \left[(\phi_{n}^{-1})^*\mu(z)\right]\, , & z\in \phi_{n}(K_n) \end{cases}
    \end{equation*}
    Since $h_n$ is harmonic in a neighborhood of $L_n$, we can find for each
    $\delta_n>0$ a
harmonic function  $u_n\colon\Omega\rightarrow \R$ such that 
$\|u_n-h_n\|_{L_n}\leq\delta_n$ (see Theorem 4 in \cite{gardiner1994superharmonic}).
 We choose $\delta_n>0$ so small that
    \begin{equation} \label{eq:est}
        \|e^{u_n}-e^{h_n}\|_{L_n}\leq \min \left\{\frac{1}{n},
          \frac{||\phi'_n||_{K_n}}{n} \right\} \, .
    \end{equation}    
    Define $\lambda_n\coloneqq e^{u_n}\in\Lambda_{0}(\Omega)$.
    Then we have $\|\lambda_n-\lambda\|_{K_n}\leq1/n$.
    On the other hand note that $\mu =\phi_{n}^*(e^{h_n})$ on $K_n$, so
    \begin{equation*}
        \left|\left|\phi_{n}^*\lambda_n-\mu\right|\right|_{K_n}=\left|\left|\phi_n^{\prime}
            \cdot  \left(e^{u_n} \circ \phi_n-e^{h_n}\circ
              \phi_n\right)\right|\right|_{K_n}\leq  ||\phi'_n||_{K_n} \cdot
        \|e^{u_n}-e^{h_n}\|_{L_n} \le \frac{1}{n}.
    \end{equation*}
    We conclude $\lambda_n \to \lambda$ and $\phi_{n}^*\lambda_n \to \mu$
    locally uniformly in $\Omega$.
\end{proof}

\subsection{Proof of Theorem \ref{thm:lu_universal_functions}}

We start with a simple observation.
Let $f \in \M(\D)$
and let $0 \in \Omega_f$, i.e.,   $0$ is a non--critical point of $f$.
Then  there is unique universal covering map $\Psi_f$ from $\D$ onto $\Omega_f$
such that  $\Psi_f(0)=0$ and $\Psi_f^{\prime}(0) > 0$. We call $\Psi_f$  the
normalized universal covering of $\Omega_f$. Note that $\Psi_f=\text{id}$ if
$f$ is locally univalent.
If $(f_n)$ is a sequence in $\M(\D)$ which converges locally $\chi$--uniformly in $\D$
to $f$, then $0$ is also a non--critical point of $f_n$ for all but finitely
many $n$, so $0 \in \Omega_{f_n}$.

\begin{prop}\label{prop:convergence_covering}
Let $f \in\mathcal{M}(\D)$ such that $0\in \Omega_f$ and suppose that
$(f_n)_n\subseteq\mathcal{M}(\D)$ converges locally $\chi$--uniformly in $\D$ to $f$.
	Then $f_n\circ\Psi_{f_n} \to f\circ\Psi_f$ locally $\chi$--uniformly in $\D$.
\end{prop}
\begin{proof} A straightforward application of Hurwitz's theorem shows that
  $\Omega_f$ is the kernel of the sequence of domains $\Omega_{f_n}$ (with
  respect to the point $0$), that is, $\Omega_f$ is the largest domain $D$ in
  $\C$ containing $0$ such that each compact subset $K$ of $D$ is contained  in all but finitely
 many of the domains $\Omega_{f_n}$.
	A well-known result of Hejhal \cite{hejhal1974universal} then implies
        $\Psi_{f_n} \to \Psi_f$ locally uniformly in $\D$.
	Since $f_n \to f$ locally $\chi$--uniformly in $\D$, we can deduce
        $f_n\circ\Psi_{f_n}  \to f\circ\Psi_f$ locally $\chi$--uniformly in
        $\D$.
\end{proof}
	
\hide{	Now suppose $f$ is constant say $f\equiv c$.
	For convenience we equip $\D$ with the hyperbolic geometry.
	Let $B_r^h(z)$ be the closed hyperbolic disc with centre $z\in\D$ and radius $r>0$.
	If we fix $0<r$ we have $\Psi_{f_n}(B_r^h(z_0))\subseteq B_r^h(w_0)$ by Schwarz-Pick's Theorem.
	Using the uniform convergence of $(f_n)_n$ on $B_r^h(w_0)$ we obtain
	\begin{equation*}
	\max_{z\in B_r^h(z_0)} |f_n\circ\Psi_{f_n}(z) - c|\leq \max_{z\in B_r^h(w_0)}|f_n(z)- c|\rightarrow 0\quad\text{for }n\rightarrow\infty.
	\end{equation*}
	This holds for any $r>0$, hence $\limn f_n\circ\Psi_{f_n} = f$. 
\end{proof}}	
\begin{proof}[Proof of Theorem \ref{thm:lu_universal_functions}]
Let $G \in \M(\D)$ be a non--constant universal function for $\mathcal{G}$.
By precomposing $G$ with a disk automorphism we may assume that $0$ is not a
critical point of $G$, so $0 \in \Omega_G$.  Let $\Psi_G$ denote the
normalized universal covering of $\Omega_G$. 
\medskip

(i) Let $F:=G \circ \Psi_G \in \M_{l.u.}(\D)$ and $f\in\mathcal{G}_{l.u.}$. We show that there is a sequence $(\phi_n)$
in $\aut{\D}$ such that
 $F \circ \phi_n \to f$ locally
$\chi$--uniformly on $\D$.

\smallskip

In fact, since $G$ is universal for $\mathcal{G}$ and $f \in \mathcal{G}$
there is a sequence $(\alpha_n)$ in $\aut{\D}$ such that $G \circ
\alpha_n \to f$ locally $\chi$--uniformly in $\D$. By our preliminary
discussion we may assume that $0 \in \Omega_{G \circ \alpha_n}$ for every $n
\in \N$. Let $\Psi_{G \circ
  \alpha_n}$ denote the normalized universal covering for $\Omega_{G \circ \alpha_n}$.
By Proposition \ref{prop:convergence_covering}, we see that
\begin{equation} \label{eq:conv0}
 G \circ \alpha_n \circ \Psi_{G \circ \alpha_n} \to f \circ \Psi_f=f \, 
\end{equation}
locally $\chi$--uniformly on $\D$. Now we observe that $\alpha_n \circ \Psi_{G
  \circ \alpha_n}$ is a universal covering map from $\D$ onto $\Omega_G$ since
$\alpha_n(\Omega_{G \circ \alpha_n})=\Omega_G$. This implies that there is
$\phi_n \in \aut{\D}$ such that 
$$ \alpha_n \circ \Psi_{G \circ \alpha_n}=\Psi_G \circ \phi_n \, .$$
Using (\ref{eq:conv0}) we get that
$$ F \circ \phi_n=G \circ \Psi_G \circ \phi_n=G \circ \alpha_n \circ \Psi_{G
  \circ \alpha_n} \to f$$
locally $\chi$--uniformly on $\D$.

\medskip

(ii) Now let $\Psi$ be any universal covering from $\D$ onto $\Omega_G$.
Then $\Psi=\Psi_G \circ T$ for some $T \in \aut{\D}$. Hence, if $f \in
\mathcal{G}_{l.u.}$ then by (i) there is a sequence $(\phi_n)$ in $\aut{\D}$
such that $G \circ \Psi_G \circ \phi_n \to f$ locally $\chi$--uniformly in
$\D$, so
$$ G \circ \Psi \circ \left( T^{-1} \circ \phi_n \right)=G \circ \Psi_G \circ \phi_n \to f$$
locally $\chi$--uniformly in $\D$ with $T^{-1} \circ \phi_n \in \aut{\D}$ for
each $n \in \N$.
\end{proof}

\subsection{Proof of Theorem  \ref{thm:u4} and Theorem \ref{thm:u3} (Case $c<0$)}

We first need to briefly recall the standard way of generating regular conformal metrics $\lambda(z)
\, |dz|$ with
constant curvature $c \in \R$ on a domain $\Omega$ in $\C$. Note that scaling $\lambda(z)
\, |dz|$ by
$|c|$, we get the metric $|c| \lambda(z) \,|dz|$, where 
 $$ c^2 \kappa_{|c|\lambda}=\kappa_{\lambda} \, .$$
In what follows we might therefore restrict ourselves without loss of generality to the normalized cases $c \in
\{-1,0,+1\}$. For these cases the canonical constantly curved metrics are
given by 
$$ \lambda_{D_c}(z)  \, |dz|:=\begin{cases}
\displaystyle \frac{2}{1-|z|^2} \, |dz| &  D_c=\D \text{ and } c=-1
\quad \text{ (hyperbolic case)\,,}\\[4mm]
\quad 1  \, |dz|  \hspace*{1.2cm}  \text{ if } & D_c=\C \text{ and } 
c=\phantom{+}0 \quad \text{ (Euclidean case)\,,} \\[2mm] 
\displaystyle \frac{2}{1+|z|^2} \, |dz| &
D_c=\hat{\C} \text{ and } c=+1\quad \text{ (spherical case\footnotemark)\,.}  
\end{cases}$$

\footnotetext{Here, we  have to consider local coordinates.}
Then for any $f \in \M_{l.u.}(\Omega)$ with $f(\Omega)
\subseteq D_c$ the pullback of $\lambda_{D_c}(z) \, |dz|$ by $f$ provides us
with a 
conformal metric $f^*\lambda_{D_c}(z) \, |dz|$ on $\Omega$ with constant
curvature $c$. Liouville \cite{liouville1853equation}  discovered that the converse
statement holds  for \textit{simply connected} domains:

\begin{thm}[Liouville]\label{thm:liouville}
 	Let $\Omega$ be a simply connected domain in $\C$ and $c \in\{-1,0,1\}$.
        Then for any  $\lambda \in\Lambda_{c}(\Omega)$.
 	there exists a function $f \in \M_{l.u.}(\Omega)$ with
        $f(\Omega)\subseteq D_c$ such that  $$\lambda  = f^*\lambda_{D_c} \, .$$
In addition, $f$ is uniquely determined by $\lambda$ up to postcomposition with
a holomorphic rigid motion $T$ of $D_c$.
\end{thm}

Recall that the holomorphic rigid motions of $D_c$ are
\begin{itemize}
\item[(a)] the conformal automorphisms of $\D$ for $c=-1$; 
\item[(b)] the direct Euclidean motions of $\C$ (i.e., the maps $z \mapsto z+b$, $b \in \C$) for $c=0$;
\item[(c)] the rotations of the Riemann sphere $\hat{\C}$ for $c=+1$.
\end{itemize}
Hence Liouville's theorem gives us, for simply connected domains $\Omega$,  a
bijection from the set of all
locally univalent holomorphic mappings from $\Omega$ to $D_c$ (modulo the action of
the rigid motions of $D_c$) onto the set
$\Lambda_c(\Omega)$ of all conformal  metrics with constant curvature $c$.
 The next result is an immediate consequence of Liouville's theorem and 
shows that this map is ``universality preserving'':


\begin{prop} \label{thm:liouvilleuniversal}
Let $\Omega$ be a simply connected domain in $\C$ and $c \in\{-1,0,1\}$.
Suppose that $\lambda \in \Lambda_c(\Omega)$ and $f \in \M_{l.u.}(\Omega)$
with $f(\Omega) \subseteq D_c$ such that $$\lambda=f^*\lambda_{D_c}\,.$$ 
If $f$ is $\Phi$--universal in $\{g \in \M_{l.u.}(\Omega) \, : \, g(\Omega)
\subseteq D_c\}$, then $\lambda$ is $\Phi$--universal in $\Lambda_c(\Omega)$.
\end{prop}

Note that Theorem \ref{thm:u4} follows directly from Proposition
\ref{thm:liouvilleuniversal} and 
\begin{itemize}
\item[(i)] Theorem \ref{thm:con_univ} if $c=0$;
\item[(ii)] Theorem \ref{thm:mer_loc_univ_birk} if $c>0$. 
\end{itemize}
Proposition \ref{thm:liouvilleuniversal} and Theorem \ref{thm:u2} also prove
the cases $c <0$ of Theorem \ref{thm:u3}.

\begin{proof}[Proof of Proposition \ref{thm:liouvilleuniversal}]
Let $\mu \in \Lambda_c(\Omega)$. By Liouville's theorem there exists a map $g \in
\M_{l.u.}(\Omega)$ with $g(\Omega) \subseteq D_c$ such that $\mu=g^*
\lambda_{D_c}$. Since $f$ is $\Phi$--universal in $\{g \in \M_{l.u.}(\Omega) \, : \, g(\Omega)
\subseteq D_c\}$ there is a sequence $(\phi_n)$ in $\Phi$ with the property
that 
$$ f \circ \phi_n \to g $$
locally $\chi$--uniformly in $\Omega$. This clearly implies
$$ \phi_n^* \lambda (z)=\lambda_{D_c} \left( (f \circ \phi_n)(z) \right)\, |
(f \circ \phi_n)'(z)| \to \lambda_{D_c} (g(z)) \, |g'(z)|=\mu(z)$$
locally uniformly in $\Omega$.
\end{proof}

\bibliographystyle{plain}

\end{document}